\documentclass{amsproc}
\usepackage{amssymb}
\usepackage[matrix, arrow, curve]{xy}
\usepackage{enumerate}
\newtheorem{theorem}{Theorem}[section]

\newtheorem{definition}[theorem]{Definition}
\usepackage[pdftex]{graphicx,color}
\theoremstyle{definition}

\def\cS{\mathcal S}
\usepackage{tikz}
\usetikzlibrary{patterns}
\tikzstyle{vertex}=[
    circle,
    minimum size = 0.1cm,
    draw=black,
    thick,
    fill=black
]

\usetikzlibrary{shapes.geometric, arrows}
\usepackage{url}
\usepackage{hyperref}
\usepackage[a4paper]{geometry}
\usepackage{pdfsync}

\numberwithin{equation}{section}

\begin{document}

\title
[]
{Combinatorics of the Quantum Symmetric Simple Exclusion Process, associahedra and free cumulants}

\author{}
\address{}


\author{Philippe  Biane}
\thanks{This research was supported by the project CARPLO, ANR-20-CE40-0007.}
\email  {biane@univ-mlv.fr }
\address{Institut Gaspard Monge
UMR CNRS - 8049
Universit\'e Gustave Eiffel
5 boulevard Descartes, 
77454 Champs-Sur-Marne
FRANCE}

\maketitle



\begin{abstract}
The Quantum Symmetric Simple Exclusion Process (QSSEP) is a model of quantum particles hopping on a finite interval and satisfying the exclusion principle. Recently Bernard and Jin have studied the fluctuations of the invariant measure for this process, when the number of sites goes to infinity. These fluctuations are encoded into polynomials, for which they have given equations and proved that these equations determine the polynomials completely. In this paper,  I give an explicit combinatorial formula for these polynomials, in terms of Schr\"oder trees. I also show that, quite surprisingly, these polynomials can be interpreted as free cumulants of a family of commuting random variables.
\end{abstract}

\maketitle

\section{Introduction}
The exclusion process is a model of particles hopping on a one-dimensional lattice, which has been extensively studied in statistical mechanics, probability, and combinatorics, see e.g. \cite{M} for an overview.
Recently, quantum versions of these processes have been introduced and the study of their fluctuations has been undertaken. In particular, Bernard and Jin \cite{BJ} have studied the invariant measure of this quantum process and shown that the asymptotic behaviour of the fluctuations of this measure, when the number of sites grows, is encoded into a family of polynomials (called ``loop polynomials'') $Q_\sigma(x_1,\ldots,x_n)$, indexed by    cyclic permutations $\sigma$  of $1,2,\ldots,n$. The first few values  of these polynomials, for permutations of small size, are given by
(we write a cyclic permutation of $1,\ldots,n$ as the sequence $1,\sigma(1),\sigma^2(1),\ldots$):
$$
\begin{array}{rcllll}
Q_1&=&x_1\\
\\
Q_{12}&=&x_1(1-x_2)\\
\\
Q_{123}&=&x_1(1-2x_2)(1-x_3)=\quad Q_{132}\\
\\
Q_{1234}&=&x_1(1-2x_3-3x_2+5x_2x_3)(1-x_4)\quad =Q_{1243}=Q_{1432}=Q_{1342}\\
\\
Q_{1324}&=&x_1(1-x_3-4x_2+5x_2x_3)(1-x_4)\quad =Q_{1423}
\end{array}
$$
They point out that the coefficients of these polynomials are integers and seem to be connected with a well known combinatorial and geometric object, the associahedron. 
In this paper we provide an explicit combinatorial expression for these polynomials. More precisely, we prove that these polynomials can be written as sums of monomials indexed by Schr\"oder trees, which are themselves in natural bijection with the faces of the associahedron. Moreover this combinatorial formula also shows that these polynomials can be interpreted as free cumulants of some family of commuting random variables. Free cumulants have been introduced in the theory of free probability by Roland Speicher, as an analogue of the cumulants appearing in probability theory. They are very efficient tools for making computations, with many applications to the theory of random matrices. See e.g. the references \cite{MS}, \cite{NS} for an overview of these fields. It is however quite surprising to see these quantities appearing in the analysis of the QSSEP.

This paper is organized as follows: in the next section I describe the physical model for the quantum symmetric simple exclusion process and how the loop polynomials arise from the analysis of the fluctuations of the two point correlation functions. In section 3, I recall the explicit characterization of the loop polynomials obtained in \cite{BJ} and  in section 4, I introduce some combinatorial notions needed to understand the explicit formula for the loop polynomials. These are   non-crossing partitions, free cumulants, Schr\"oder trees, the cluster complex and the associahedra. The connections between Schr\"oder trees and free cumulants have been explored in \cite{JMNT} (see also \cite{CN}, sect. 1.2) and we use some of their results. Section 5 is devoted to stating and proving  the combinatorial formula for  the loop polynomials. Finally, in section 6, I show that the loop polynomials can also be interpreted as  free cumulants of a family of commuting random variables.

I would like to thank Denis Bernard for pointing out this problem to me as well as for all  his  explanations, Matthieu Josuat-Verg\`es for directing me to the paper \cite{JMNT} and the referees for their comments and suggestions.
\section{The model}\label{model}
In this section I describe briefly the physical model at the origin of the construction of the loop polynomials. A more precise description can be found  in \cite{BJ}, to which I refer for further information.
My aim here is just to explain the physical context in which the loop polynomials appear,
readers interested only in the combinatorial aspect of the problem can jump directly to the next section, where I give the precise algebraic definition of the loop polynomials.

Let  $c_i^\dagger,c_i$, indexed by integers $i=0,1,\ldots, N-1$, be a family of fermionic creation and annihilation operators, satisfying the usual anticommutation relations
$$c_ic_j^\dagger+c_j^\dagger c_i=\delta_{ij}$$
These operators are defined on a Hilbert space $V$, of dimension $2^N$, and describe a system of fermionic particles situated at the points $0,1,\ldots, N-1$. The quantum symmetric exclusion process has a Hamiltonian $H_t$, which is a random selfadjoint operator on $V$,  depending on the time $t\geq 0$ and satisfying a stochastic differential equation
$$dH_t=\sum_{j=0}^{N-1}c_{j+1}^\dagger c_jdW_t^j+c_j^\dagger c_{j+1}d\bar W_t^j$$ where $W_t^j,j=0,\ldots N-1$ are a family of independent complex Brownian motions. The Hamiltonian describes the motions of the quantum particles which, due to Brownian noise, can jump from their site to one of the  nearest neighbouring sites. It is supplemented by boundary conditions at the points $0$ and $N-1$, where the system is in contact with a reservoir, with which it can exchange particles.
From the Hamiltonian one deduces the equation of motion for the density matrix $\rho_t$, a positive self-adjoint matrix on $V$, with $Tr(\rho_t)=1$:
$$d\rho_t=-i[dH_t,\rho_t]-\frac{1}{2}[dH_t,[dH_t,\rho_t]]+\mathcal L_{bdry}(\rho_t)dt$$
Here $\mathcal L_{bdry}$ is a linear operator of Linblad form describing the boundary conditions at the points  $0$ and $N-1$. As $t\to \infty$ the density matrix converges in distribution to a random  density matrix $\rho$ which represents the stationary distribution of the process.

The quantities of interest are the fermion two-point functions in the steady measure, namely
$$G_{ij}=Tr(c_j^\dagger c_i\rho),$$ which are random variables,  and their connected correlation functions (known as cumulants in mathematical language) defined, for families of integers $i_1,i_2,\ldots,i_p,j_1,j_2,\ldots j_p\in [0,N-1]$ by
$$E[G_{i_1j_1}G_{i_2j_2}\ldots G_{i_pj_p}]^c$$
In the limit of $N\to\infty$ the leading cumulants scale as $N^{-p+1}$ and only the ones for which $j_1,\ldots, j_p$ is a cyclic permutation of $i_1,\ldots,i_p$ have a nonzero limit. More precisely, if $i_1/N,i_2/N,\ldots ,i_p/N$ converge to real numbers  $u_1,u_2, \ldots , u_p\in [0,1]$ as $N\to\infty$, then 
$$E[G_{i_1i_p}G_{i_pi_{p-1}}\ldots G_{i_2i_1}]^c=\frac{1}{N^{p-1}}g_p(u_1,\ldots,u_p)+O(\frac{1}{N^{p}})$$
for some functions $g_p$. The $g_p$ turn out to be piecewise polynomial functions, being given by a certain polynomial in each sector corresponding to an ordering of the $u_i$. In order to study the  $g_p$ it is thus convenient to consider their restrictions to these different sectors. 
Specifically, introduce functions $Q_\sigma(x_1,\ldots,x_p)$ defined for $0\leq x_1\leq x_2\leq\ldots\leq x_p\leq 1$, indexed by circular permutations $\sigma $ of $1,\ldots, p$ which satisfy:

\begin{equation}\label{Q}
E[G_{i_1i_{\sigma^{p-1}(1)}}G_{i_{\sigma^{p-1}(1)}i_{\sigma^{p-2}(1)}}\ldots G_{i_{\sigma(1)}i_{1}}]^c=\frac{1}{N^{p-1}}Q_\sigma(x_1,\ldots,x_p)+O(\frac{1}{N^{p}}).
\end{equation}
Here we assume that the  limits of the $i_k/N$ exist as $N\to\infty$ and correspond, once ordered, to the points $x_1,\ldots,x_p$. The permutation $\sigma$ encodes the order in which the $x_i$ occur in the left-hand side of (\ref{Q}).
The $Q_\sigma$ are polynomials and they are determined by some  algebraic conditions which we describe in the next section. While the algebraic definition of the loop polynomials is unambiguous, the fact that they satisfy (\ref{Q}) is not proved in full generality and still partly conjectural. It would be nice to have a complete proof of this fact.
\section{Characterization of the loop polynomials}
In this section I give the characterization of the loop polynomials $Q_\sigma$ appearing in formula (\ref{Q}), following \cite{BJ}. Whereas their original definition is valid when the $x_i$ form a set of ordered real numbers in the interval $[0,1]$, the polynomials $Q_\sigma$ can be defined purely algebraically as follows.
\begin{definition}\label{loop}
 The loop polynomials form  a family of polynomials $Q_\sigma(x_1,\ldots,x_n)$, indexed by circular permutations of  $1,\ldots,n$, with the following properties.

\begin{enumerate}
\item
$Q_\sigma$ is of degree 1 in each of the variables $x_i$.
\item
If $n=1$ then $Q_{1}(x_1)=x_1$, if $n=2$ then $Q_{12}(x_1,x_2)=x_1(1-x_2)$.

\item
For $n\geq 2$ they satisfy the boundary conditions:
$$Q_\sigma=x_1P_\sigma(x_2,\ldots,x_{n-1})(1-x_n)$$ where $P_\sigma$ is a polynomial in $n-2$ variables.
\item
For $i=1,2,\ldots, n-1$ they satisfy a continuity condition:
$$Q_\sigma|_{x_i=x_{i+1}}=Q_{s_i\sigma s_i}|_{x_i=x_{i+1}}$$
\item For $i=1,2,\ldots, n-1$ they satisfy an exchange relation:
$$\left([x_i](Q_\sigma+Q_{s_i\sigma s_i})\right)|_{x_i=x_{i+1}}-\left([x_{i+1}](Q_\sigma+Q_{s_i\sigma s_i})\right)|_{x_i=x_{i+1}}=2([x_i]Q_{\sigma^-}(x^-))([x_{i+1}]Q_{\sigma^+}(x^+))$$
\end{enumerate}
Here $s_i$ is the transposition $(i\, i+1)$ and, if $P$ is a polynomial in several variables, $[x]P$ denotes the coefficient of the monomial $x$ in $P$. It is a polynomial in the remaining variables. Since $\sigma$ is a cycle, the permutation $s_i\sigma=\sigma^+\sigma^-$ is the product of two disjoint cycles, denoted by $\sigma^+$, which moves  $i+1$ and
$\sigma^-$, which moves $i$. The variables denoted by $x^-$ are those with the indices moved by $\sigma^-$ and similarly for $x^+$.
\end{definition}

Using the relations in the definition above, the authors prove in \cite{BJ}
 that the polynomials $Q_\sigma$ are uniquely defined, moreover they show that  the $Q_\sigma$, for $\sigma=(123\ldots n)$, can be computed by induction on $n$.

Let us explain in more details, following \cite{BJ}, how conditions (3), (4), (5)  can be used to compute the loop polynomials. Fix some integer $i\in [1,n-1]$. Since the polynomials have degree one in each variable, it is possible to expand $Q_\sigma$ as 
$$Q_\sigma=A+x_iB+x_{i+1}C+x_ix_{i+1}D,$$ where the polynomials $A,B,C,D$ do not depend on $x_i,x_{i+1}$. If we do the same for $Q_{s_i\sigma s_i}$, namely
$$Q_{s_i\sigma s_i}=A'+x_iB'+x_{i+1}C'+x_ix_{i+1}D',$$
then (4) implies that $$A=A',\quad D=D',\quad B+C=B'+C'.$$ Let us consider the factorization $s_i\sigma=\sigma^-\sigma^+$ and  the polynomial 
$$\Delta:=([x_i]Q_{\sigma^-}(x^-))([x_{i+1}]Q_{\sigma^+}(x^+)).$$
The continuity (4) and exchange (5) conditions   imply that 
\begin{equation}\label{delta}
B-C'=B'-C=\Delta.
\end{equation}

Since all cyclic permutations are conjugated these relations allow thus to compute $Q_\sigma$ for any cycle of length $n$, if we know the value at one particular cycle of length $n$ and the values for cycles of smaller length. Since the value of $Q_{123\ldots n}$ can be computed by induction on $n$, this shows that the polynomials $Q_\sigma$ are uniquely defined. It is not clear a priori, from these purely algebraic properties, that the polynomials exist, since there might be too many constraints on them, in particular it is not clear why the boundary conditions (3) will be fulfilled, but our explicit formula below will provide a direct proof that, indeed, polynomials satisfying all these properties exist.

As an example let $\sigma=(1234)$ and $i=2$ so that $s_i=s_2=(23)$ and $s_2(1234)s_2=(1324)$. One has
 $$Q_\sigma=x_1(1-3x_2-2x_3+5x_2x_3)(1-x_4)$$ so that
$A=x_1$, $B=-3x_1(1-x_4)$, $C=-2x_1(1-x_4)$, $D=5x_1(1-x_4)$, moreover  $$s_2\sigma=(134)$$ so that $\sigma^-=(2), x^-=\{x_2\}$ and $\sigma^+=(134), x^+=\{x_1,x_3,x_4\}$. Thus 
$$
\begin{array}{ll}
Q_{\sigma^-}=x_2, & [x_2]Q_{\sigma^-}=1\\
Q_{\sigma^+}=x_1(1-2x_3)(1-x_4),& [x_3]Q_{\sigma^+}=-2x_1(1-x_4)
\end{array}
$$
 It follows that $\Delta=-2x_1(1-x_4)$, therefore, by equation (\ref{delta}), one has 
$B'=C+\Delta=-4x_1(1-x_4)$ and $C'=B-\Delta=-x_1(1-x_4)$. Thus, from the knowledge of $Q_{1234}$ we find the polynomial
$$Q_{1324}=x_1(1-4x_2-x_3+5x_2x_3)(1-x_4).$$ 
Here are a few more values, for
 cyclic  permutations of length 5, when  the polynomials can take four different values:

$$
\begin{array}{rcllll}
Q_{12345}&=&x_1(1-4x_2-3x_3-2x_4+9x_2x_3+7x_2x_4+5x_3x_4-14x_2x_3x_4)(1-x_5)
\\
Q_{13245}&=&x_1(1-6x_2-x_3-2x_4+9x_2x_3+10x_2x_4+2x_3x_4-14x_2x_3x_4)(1-x_5)\\
Q_{12435}&=&x_1(1-4x_2-4x_3-x_4+12x_2x_3+4x_2x_4+5x_3x_4-14x_2x_3x_4)(1-x_5)\\
Q_{14235}&=&x_1(1-6x_2-2x_3-x_4+12x_2x_3+7x_2x_4+2x_3x_4-14x_2x_3x_4)(1-x_5)
\end{array}
$$

Indeed one can check that
$$\begin{array}{l}
Q_{12345}=Q_{13452}=Q_{14523}=Q_{15234}=Q_{15432}=Q_{12543}=Q_{13254}=Q_{14325}\\
Q_{13245}=Q_{13254}=Q_{15423}=Q_{14523}\\
Q_{12435}=Q_{14352}=Q_{15342}=Q_{12534}\\
Q_{14235}=Q_{13524}=Q_{15234}=Q_{15324}=Q_{14253}=Q_{13425}=Q_{15243}=Q_{14325}
\end{array}
$$
All these examples suggest  a combinatorial significance of the loop polynomials. In particular, as pointed out in \cite{BJ},  the sum of the absolute values of the coefficients are Schr\"oder numbers (whose definition is recalled in the next section). I will give  a combinatorial formula for the loop polynomials in section 5, after introducing the objects necessary to describe it in the next section.

\section{Some combinatorial objects}
This section is devoted to a description of the combinatorial objects which will be used for expressing the loop polynomials.
\subsection{Cumulants}
First we recall briefly the classical theory of cumulants.
In probability theory the {\sl cumulants} are multilinear expressions $C_n(a_1,\ldots,a_n)$ in random variables which arise in the expansion of the free energy as

\begin{equation}\label{free-energy}
\log E[e^{\sum_{i=1}^N\lambda_ia_i}]=\sum_{n=1}^\infty\sum_{i_1+\ldots+i_N=n}\frac{\lambda_1^{i_1}}{i_1!}\cdots\frac{\lambda_1^{i_N}}{i_N!}C_n(a_I)
\end{equation}
where $E$ denotes the expectation and $a_I=(a_1,a_1,\ldots,a_1,a_2,\ldots,a_2,\ldots, ,a_N,\ldots a_N)$ with $i_k$ occurrences of $a_k$.

There is a combinatorial way to define the cumulants using the set-partitions of $\{1,\ldots,n\}$, which form a lattice ${\mathcal P}_n$, via the implicit formula:

\begin{equation}\label{mom-cum}
E[a_1\ldots a_n]=\sum_{\pi\in {\mathcal P}_n} C_\pi(a_1,\ldots,a_n)
\end{equation} 
where
\begin{equation}\label{mom-cum1}
C_\pi(a_1,\ldots,a_n)=\prod_{p\in \pi}C_{|p|}(a_{i_1},\ldots,a_{i_{|p|}})
\end{equation} 
the product being over the parts $p$ of $\pi$ with 
 $p=\{i_1,\ldots,i_{|p|}\}$ and $i_1<i_2<\ldots<i_{|p|}$.
This formula can be inverted to express the cumulants in terms of the ``moments'', i.e. $E$ evaluated on products of the $a_i$.

$$C_n(a_1,\ldots,a_n)=\sum_{\pi\in {\mathcal P}_n} E_\pi(a_1,\ldots, a_n)(-1)^{|\pi|-1}(|\pi|-1)!$$

\subsection{Non-crossing partitions and free cumulants}
\subsubsection{Non-crossing partitions}
A set partition of $\{1,2,\ldots,n\}$ (or any other totally ordered set) is called {\sl non-crossing} if there is no quadruple $i,j,k,l$ such that $i<j<k<l$ while $i,k$ belong to some part of the partition and $j,l$ belong to another part. One can picture a non-crossing partition by putting the points $1,\ldots,n$ in cyclic order on a circle, and drawing, for each part of the partition,  the convex polygon whose vertices are the elements of the part. The partition is non-crossing if and only if these polygons are disjoint. This construction shows, in particular, that if $\pi$ is a non-crossing partition then its image by the cyclic permutation $(123\ldots n)$ is still non-crossing.
For example, here is the non-crossing partition $\pi=\{1,3,4\}\cup\{2\}\cup\{5,6\},\cup\{7\}\cup\{8\}$

$$
\begin{tikzpicture}[scale=1]

\draw (0,0) circle (2cm);
\draw[fill] (0,2) circle (.1cm);\draw[fill] (0,-2) circle (.1cm);\draw[fill] (2,0) circle (.1cm);\draw[fill] (-2,0) circle (.1cm);\draw[fill] (1.41,1.41) circle (.1cm);
\draw[fill] (1.41,-1.41) circle (.1cm);\draw[fill] (-1.41,1.41) circle (.1cm);\draw[fill] (-1.41,-1.41) circle (.1cm);
\draw (0,2)--(2,0)--(1.41,-1.41)--(0,2);
\draw (0,-2)--(-1.41,-1.41);
\node at (0,2.4){$1$};\node at (1.7,1.7){$2$};\node at (2.4,0){$3$};\node at (1.7,-1.7){$4$};\node at (0,-2.4){$5$};\node at (-1.7,-1.7){$6$};\node at (-2.4,0){$7$};\node at (-1.7,1.7){$8$};
\end{tikzpicture}
$$
We denote by $NC(n)$ the set of non-crossing partitions of $\{1,2,\ldots,n\}$.
The  non-crossing partitions are counted by the Catalan numbers: there are $\text{Cat}_n=\frac{1}{n+1}{2n \choose n}$ non-crossing partitions of $\{1,2,\ldots,n\}$. Moreover, the set $NC(n)$ endowed with the reverse refinement order is a lattice, it is ranked by $rk(\pi)=n-|\pi|$ where $|\pi|$ is the number of parts of $\pi$.
I refer to \cite{NS} for an in-depth  study of the set of non-crossing partitions and for   the proofs of all statements about non-crossing partitions. 

\subsubsection{Kreweras complement and the M\"obius function}
The {\sl Kreweras complement} of a non-crossing partition is obtained by putting primed points $1',2',\ldots$ between consecutive  points on the circle and forming the largest possible polygons, with primed vertices, which do not cross the polygons of the original partition. With our example above we get the partition $K(\pi)=\{1,5,7,8\}\cup\{2,3\}\cup\{4\}\cup \{6\}$ (where we have removed the primes), see the picture below.
$$
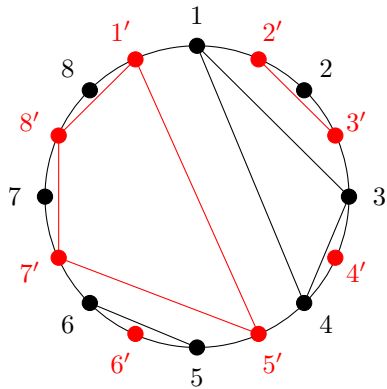
\begin{tikzpicture}[scale=1]

\draw (0,0) circle (2cm);
\draw[fill] (0,2) circle (.1cm);\draw[fill] (0,-2) circle (.1cm);\draw[fill] (2,0) circle (.1cm);\draw[fill] (-2,0) circle (.1cm);\draw[fill] (1.41,1.41) circle (.1cm);
\draw[fill] (1.41,-1.41) circle (.1cm);\draw[fill] (-1.41,1.41) circle (.1cm);\draw[fill] (-1.41,-1.41) circle (.1cm);
\draw (0,2)--(2,0)--(1.41,-1.41)--(0,2);
\draw (0,-2)--(-1.41,-1.41);
\node at (0,2.4){$1$};\node at (1.7,1.7){$2$};\node at (2.4,0){$3$};\node at (1.7,-1.7){$4$};\node at (0,-2.4){$5$};\node at (-1.7,-1.7){$6$};\node at (-2.4,0){$7$};\node at (-1.7,1.7){$8$};

\draw[fill, color=red] (0.81,1.82) circle (.1cm);\draw[fill, color=red] (-0.81,1.82) circle (.1cm);\draw[fill, color=red] (-0.81,-1.82) circle (.1cm);\draw[fill, color=red] (0.81,-1.82) circle (.1cm);
\draw[fill, color=red] (1.82,0.81) circle (.1cm);\draw[fill, color=red] (1.82,-0.81) circle (.1cm);\draw[fill, color=red] (-1.82,-0.81) circle (.1cm);\draw[fill, color=red] (-1.82,0.81) circle (.1cm);
\node[color=red] at (-1,2.2){$1'$};\node[color=red] at (1,2.2){$2'$};
\node[color=red] at (2.1,1){$3'$};
\node[color=red] at (-1,-2.2){$6'$};\node[color=red] at (1,-2.2){$5'$};\node[color=red] at (2.1,-1){$4'$};
\node[color=red] at (-2.2,-1){$7'$};\node[color=red] at (-2.2,1){$8'$};
\draw[color=red] (-0.81,1.82)--(0.81,-1.82)--(-1.82,-0.81)--(-1.82,0.81)--(-0.81,1.82);
\draw[color=red] (0.81,1.82)--(1.82,0.81);
\end{tikzpicture}
$$

 The Kreweras complement is an anti-isomorphism for the order on $NC(n)$.
Note that our definition of the Kreweras complement does not coincide with that of \cite{NS}, rather, it corresponds to the inverse of the Kreweras complement of \cite{NS}. We use this version of the Kreweras map because it is slightly easier to use in our computations.

Recall that, for a partially ordered set, its {\sl zeta function} is defined as $\zeta(x,y)=1$ if $x\leq y$ and $\zeta(x,y)=0$ if not.
The {\sl M\"obius function} satisfies $\mu(x,y)=0$ unless $x\leq y$ and, for $x\leq z$:
$$\sum_{y:x\leq y\leq z}\mu(x,y)\zeta(x,z)=\delta_{xz}$$
If we consider $\zeta$ as an upper triangular matrix encoding the order relation, then $\mu$ is the inverse matrix.

We will  consider the M\"obius function of $NC(n)$ and put $\mu(\pi):=\mu(\pi,1_n)$ where $1_n$ is the partition with one part.
 The  function $\mu$ can be expressed, using the Kreweras complement, as
\begin{equation}\label{Mob}
\mu(\pi)=\prod_{\text{$p$ part of $K(\pi)$}}(-1)^{|p|-1}\text{Cat}_{|p|-1}
\end{equation}

As recalled in the Introduction, free cumulants have been introduced by Roland Speicher in the theory of free probability, as analogues of cumulants in probability theory. The setting is the following: we consider a unital algebra $A$, over some field $k$, which is usually taken, for applications to probability,  to be the complex numbers, although the theory can be developed, for a large part, within a purely algebraic framework. This algebra is endowed with a linear form $\varphi$ such that $\varphi(1)=1$. 
The free cumulants  form a family $\kappa_n,n=1,2,3,\ldots$ of multilinear forms on $A$, such that $\kappa_n$ is a $n$-linear form. Define, for 
 a non-crossing partition $\pi$ of $[1,n]$  a $n$-linear form $\kappa_\pi$ on $A$ by
$$\kappa_\pi(a_1,\ldots,a_n)=\prod_{\text{$p$ part of $\pi$}}\kappa_{|p|}(a_{i_1},a_{i_2},\ldots ,a_{i_{|p|}})$$
Here the product is over the parts $p$ of the partition $\pi$ and, for a part $p$, we denote by $|p|$ its size and by $i_1,\ldots ,i_{|p|}$ its elements, listed in increasing order.
The following relation defines implicitly the free cumulants as the only sequence $\kappa_n$ satisfying, for all $n$ and all $a_1,\ldots,a_n\in A$:
\begin{equation}\label{moment-cumulant}
\varphi(a_1a_2\ldots a_n)=\sum_{\pi\in NC(n)}\kappa_\pi(a_1,a_2,\ldots,a_n).
\end{equation}
 It is analogous to the combinatorial formula for cumulants (\ref{mom-cum}).
Using the M\"obius function of $NC(n)$, one   can invert the relation (\ref{moment-cumulant}) and  express the free cumulants  explicitly  as
\begin{equation}\label{cumulant-moment}
\kappa_n(a_1,a_2,\ldots,a_n)=\sum_{\pi\in NC(n)}\varphi_\pi(a_1,\ldots ,a_n)\mu(\pi).
\end{equation}
Here $\varphi_\pi$ is defined in terms of $\varphi$ by a formula similar to that for  the $\kappa_\pi$:
$$\varphi_\pi(a_1,\ldots,a_n)=\prod_{\text{$p$ part of $\pi$}}\varphi(a_{i_1}a_{i_2}\ldots a_{i_{|p|}}).$$
\subsection{Schr\"oder trees and associahedra}
\subsubsection{Schr\"oder trees}
{\sl Schr\"oder trees} are plane, rooted trees such that each internal vertex has at least two descendants. Such trees are counted, in terms of the number of leaves, by the small Schr\"oder numbers $s_n=1,1,3,11,45,\ldots$ for $n=1,2,3,\ldots$ with generating series
$\frac{1 + x - \sqrt{1 - 6x + x^2}}{4x}$ (the case $n=1$ is special in that the  root is not an internal vertex, it is a leaf). These numbers form the sequence  A001003 in \cite{OEIS}. The set of binary trees is a subset, counted again by Catalan numbers: there are $\text{Cat}_{n-1}$ binary trees with $n$ leaves. 
Figure 1 shows the Schr\"oder trees with four leaves. The leaves are the white vertices, the internal vertices are black and the root is denoted by a square. The upper row shows the binary trees, which have the maximal number of vertices (there are $\text{Cat}_3=5$ of them in the case of four leaves shown here). For each binary tree we can contract some of its left internal edges (an internal edge is an edge joining two internal vertices) in order to get a Schr\"oder tree. Below each binary tree I show the Schr\"oder trees obtained in this way. The first column has four trees, corresponding to the upper left binary tree with two left internal edges and the trees obtained by contracting these  edges. The next three columns correspond to binary trees with one left internal edge.  Finally, the last binary tree has no left internal edge.

\begin{figure}[h!]
\begin{tikzpicture}[scale=.55]
\draw[fill] (2.8,-.2)--(2.8,.2)--(3.2,.2)--(3.2,-.2)--(2.8,-.2);\draw (1.5,1) circle (.2cm);\draw (2.5,1) circle (.2cm);\draw (3.5,1) circle (.2cm);\draw (4.5,1) circle (.2cm);
\draw (3,0)--(1.7,0.85);\draw (3,0)--(2.6,0.8);\draw (3,0)--(3.4,0.8);\draw (3,0)--(4.4,0.8);

\draw[fill] (2.8,1.8)--(2.8,2.2)--(3.2,2.2)--(3.2,1.8)--(2.8,1.8);\draw[fill] (2,3) circle (.2cm);\draw (1,4) circle (.2cm);\draw (3,3) circle (.2cm);\draw (4,3) circle (.2cm);\draw (3,4) circle (.2cm);
\draw (2,3)--(3,2);\draw (2,3)--(1.15,3.85);\draw (3,2)--(3,2.8);\draw (3,2)--(3.84,2.84);\draw (2,3)--(2.85,3.85);

\draw[fill] (2.8,4.8)--(2.8,5.2)--(3.2,5.2)--(3.2,4.8)--(2.8,4.8);\draw[fill] (2,6) circle (.2cm);\draw (1,7) circle (.2cm);\draw (2,7) circle (.2cm);\draw (3,7) circle (.2cm);\draw (4,6) circle (.2cm);
\draw (3,5)--(2,6);\draw (2,6)--(1.15,6.85);\draw (2,6)--(2,6.85);\draw (2,6)--(2.84,6.84);\draw (3,5)--(3.85,5.85);

\draw[fill] (2.8,7.8)--(2.8,8.2)--(3.2,8.2)--(3.2,7.8)--(2.8,7.8);\draw[fill] (2,9) circle (.2cm);\draw[fill] (1,10) circle (.2cm);
\draw (0,11) circle (.2cm);\draw (2,11) circle (.2cm);\draw (3,10) circle (.2cm);\draw (4,9) circle (.2cm);
\draw (3,8)--(.15,10.85);\draw (3,8)--(3.85,8.85);\draw (2,9)--(2.85,9.85);\draw (1,10)--(1.85,10.85);

\draw[fill] (7.8,7.8)--(7.8,8.2)--(8.2,8.2)--(8.2,7.8)--(7.8,7.8);\draw[fill] (7,9) circle (.2cm);\draw[fill] (8,10) circle (.2cm);
\draw (8,8)--(7,9)--(8,10);\draw (6,10) circle (.2cm);\draw (7,11) circle (.2cm);\draw (9,11) circle (.2cm);\draw (9,9) circle (.2cm);
\draw (7,9)--(6.15,9.85);\draw (8,10)--(7.15,10.85);\draw (8,10)--(8.85,10.85);\draw (8,8)--(8.85,8.85);

\draw[fill] (12.8,7.8)--(12.8,8.2)--(13.2,8.2)--(13.2,7.8)--(12.8,7.8);\draw[fill] (12,9) circle (.2cm);\draw[fill] (14,9) circle (.2cm);
\draw (12,9)--(13,8)--(14,9);\draw (11,10) circle (.2cm);\draw (12.3,10) circle (.2cm);\draw (13.7,10) circle (.2cm);\draw (15,10) circle (.2cm);
\draw (12,9)--(11.15,9.85);\draw (12,9)--(12.2,9.8);\draw (14,9)--(13.8,9.8);\draw (14,9)--(14.85,9.85);

\draw[fill] (17.8,7.8)--(17.8,8.2)--(18.2,8.2)--(18.2,7.8)--(17.8,7.8);\draw[fill] (18,10) circle (.2cm);\draw[fill] (19,9) circle (.2cm);
\draw (18,8)--(19,9)--(18,10);\draw (17,9) circle (.2cm);\draw (17,11) circle (.2cm);\draw (19,11) circle (.2cm);\draw (20,10) circle (.2cm);
\draw (18,8)--(17.15,8.85);\draw (18,10)--(17.15,10.85);\draw (18,10)--(18.85,10.85);\draw (19,9)--(19.85,9.85);

\draw[fill] (22.8,7.8)--(22.8,8.2)--(23.2,8.2)--(23.2,7.8)--(22.8,7.8);\draw[fill] (24,9) circle (.2cm);\draw[fill] (25,10) circle (.2cm);\draw (23,8)--(25,10);
\draw (26,11) circle (.2cm);\draw (24,11) circle (.2cm);\draw (23,10) circle (.2cm);\draw (22,9) circle (.2cm);
\draw (23,8)--(22.15,8.85);\draw (24,9)--(23.15,9.85);\draw (25,10)--(24.15,10.85);\draw (25,10)--(25.85,10.85);

\draw[fill] (7.8,4.8)--(7.8,5.2)--(8.2,5.2)--(8.2,4.8)--(7.8,4.8);\draw[fill] (8,6) circle (.2cm);
\draw (8,5)--(8,6);\draw (7,6) circle (.2cm);\draw (7,7) circle (.2cm);\draw (9,6) circle (.2cm);\draw (9,7) circle (.2cm);
\draw (8,5)--(7.15,5.85);\draw (8,5)--(8.85,5.85);\draw (8,6)--(7.15,6.85);\draw (8,6)--(8.85,6.85);

\draw[fill] (12.8,4.8)--(12.8,5.2)--(13.2,5.2)--(13.2,4.8)--(12.8,4.8);\draw[fill] (14,6) circle (.2cm);
\draw (13,5)--(14,6);\draw (12,6) circle (.2cm);\draw (13,6) circle (.2cm);\draw (13,7) circle (.2cm);\draw (15,7) circle (.2cm);
\draw (13,5)--(12.15,5.85);\draw (13,5)--(13,5.8);\draw (14,6)--(13.15,6.85);\draw (14,6)--(14.85,6.85);

\draw[fill] (17.8,4.8)--(17.8,5.2)--(18.2,5.2)--(18.2,4.8)--(17.8,4.8);\draw[fill] (19,6) circle (.2cm);
\draw (18,5)--(19,6);\draw (17,6) circle (.2cm);\draw (18,7) circle (.2cm);\draw (19,7) circle (.2cm);\draw (20,7) circle (.2cm);
\draw (18,5)--(17.15,5.85);\draw (19,6)--(18.15,6.85);\draw (19,6)--(19,6.8);\draw (19,6)--(19.85,6.85);

\end{tikzpicture}
\caption{ Schr\"oder trees with four leaves}
\end{figure}
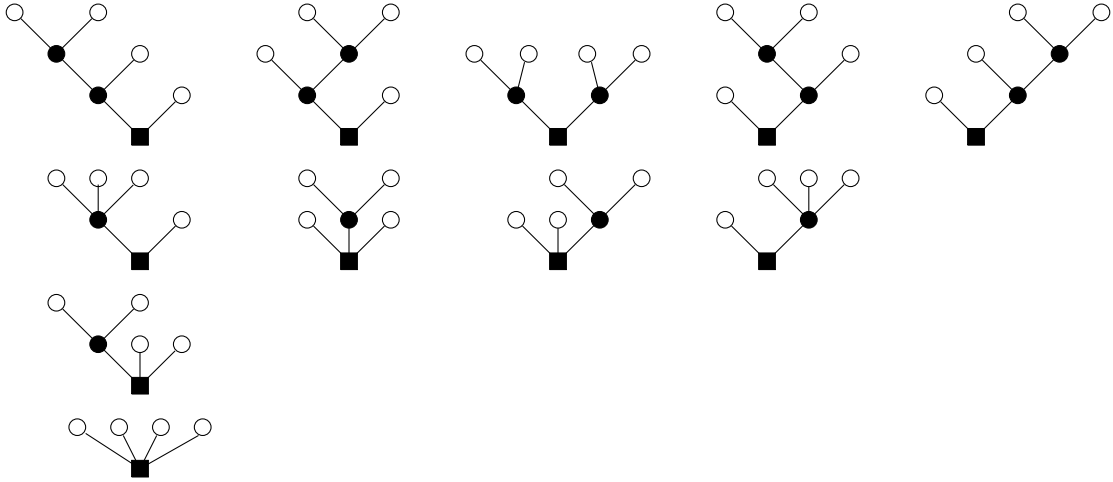
\subsubsection{Prime Schr\"oder trees}
A Schr\"oder tree is called {\sl prime} if the righmost edge of its root is a leaf. The prime Schr\"oder trees are counted by the large Schr\"oder numbers $S_n=2s_{n-1}$. Indeed given a Schr\"oder tree 
$t$, with $n-1$  leaves, we can build two prime Schr\"oder trees $t_1$ and $t_2$, with $n$ leaves: either one appends a new leaf to the right of the root of the tree to get $t_1$, or one builds a new tree by putting the Schr\"oder tree on the left branch of a binary tree with two leaves to get $t_2$, see the picture below. 
$$
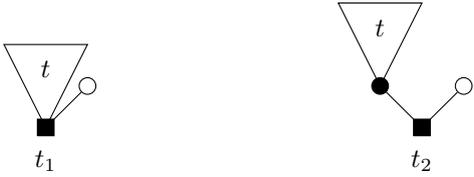
\begin{tikzpicture}[scale=.55]
\draw[fill] (2.8,-.2)--(2.8,.2)--(3.2,.2)--(3.2,-.2)--(2.8,-.2);
\draw (3,0)--(2,2)--(4,2)--(3,0);
\node at (3,1.4) {$t$};
\draw (4,1) circle (.2cm);
\draw (3,0)--(3.85,.85);
\draw[fill] (11.8,-.2)--(11.8,.2)--(12.2,.2)--(12.2,-.2)--(11.8,-.2);
\draw (12,0)--(11,1);
\draw[fill] (11,1) circle (.2cm);
\draw (11,1)--(10,3)--(12,3)--(11,1);
\node at (11,2.4) {$t$};
\draw (13,1) circle (.2cm);
\draw (12,0)--(12.85,.85);
\node at (3,-.8) {$t_1$};
\node at (12,-.8) {$t_2$};
\end{tikzpicture}
$$
It is easy to see that each prime Schr\"oder tree is uniquely obtained in one of these  ways from a Schr\"oder tree with one leaf less.
We denote by $\cS_n$ the set of Schr\"oder trees with $n$ leaves and by $p\cS_n\subset \cS_n$ that of prime Schr\"oder trees.
There are $6$ prime Schr\"oder trees with $4$ leaves and they are in the first two columns of Figure 1.

\subsubsection{Corners}
 A {\sl corner} in a Schr\"oder tree is a pair of edges leading to  consecutive descendants (from left to right) of an internal vertex. Such a corner is said to belong to the vertex from which the descendants are originated. A Schr\"oder tree with $n$ leaves has exactly $n-1$ corners. 
Let $t$ be a  a  Schr\"oder tree with $n$ leaves. Label the corners of the tree from left to right by the numbers $1,2,\ldots, n-1$, in a contour exploration of the tree. Here is an example, with a prime Schr\"oder tree having $6$ internal vertices and $11$ corners:
$$
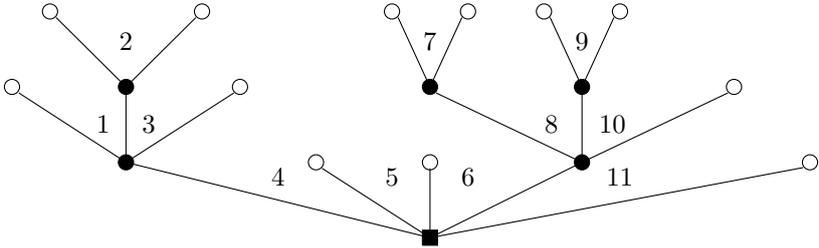
\begin{tikzpicture}[scale=1]
\draw[fill] (5.9,-.1)--(6.1,-.1)--(6.1,.1)--(5.9,.1)-- (5.9,-.1);\draw[fill] (2,1) circle (.1cm);
\draw[fill] (2,2) circle (.1cm);\draw[fill] (8,1) circle (.1cm);\draw[fill] (6,2) circle (.1cm);\draw[fill] (8,2) circle (.1cm);

\draw (4.5,1) circle (.1cm);\draw (6,1) circle (.1cm);\draw (11,1) circle (.1cm);
\draw (0.5,2) circle (.1cm);\draw (3.5,2) circle (.1cm);\draw(10,2) circle (.1cm);
\draw (1,3) circle (.1cm);\draw (3,3) circle (.1cm);\draw (5.5,3) circle (.1cm);\draw (6.5,3) circle (.1cm);\draw (7.5,3) circle (.1cm);\draw (8.5,3) circle (.1cm);

\draw (2,2)--(1.08,2.92);\draw (2,2)--(2.92,2.92);
\draw (.59,1.92)--(2,1)--(6,0)--(10.91,0.93);\draw (2,1)--(2,2);\draw (6,0)--(8,1);
\draw (2,1)--(3.42,1.92);\draw (6,0)--(4.58,.92);\draw (6,0)--(6,.92);
\draw (8,1)--(6.08,1.92);\draw (8,1)--(8,1.92);\draw (8,1)--(9.92,1.92);
\draw (6,2)--(5.58,2.92);\draw (6,2)--(6.42,2.92);\draw (8,2)--(8.42,2.92);\draw (8,2)--(7.58,2.92);

\node at (1.7,1.5) {$1$};\node at (2,2.6) {$2$};\node at (2.3,1.5) {$3$};\node at (4,.8) {$4$};\node at (5.5,.8) {$5$};\node at (6.5,.8) {$6$};\node at (6,2.6) {$7$};
\node at (7.6,1.5) {$8$};\node at (8,2.6) {$9$};\node at (8.4,1.5) {$10$};\node at (8.5,0.8) {$11$};
\end{tikzpicture}
$$

Consider the partition of $\{1,2,\ldots,n-1\}$ such that $i,j$ are in the same part if and only if they label corners which belong to the same vertex. 
It is easy to see that this partition is non-crossing and  we denote it by $\pi(t)$. The number of parts of $\pi(t)$ is equal to the number of internal vertices of $t$.
In our example the partition is formed by the sets
$$\{1,3\},\{2\},\{4,5,6,11\},\{7\},\{8,10\},\{9\}$$
There is a natural operad built on Schr\"oder trees,  which has  been used in  \cite{JMNT} to recover the M\"obius function on $NC(n)$. I will now use some of their constructions.
Consider a prime  Schr\"oder tree. One obtains a 
  forest  from the  tree by removing, for each internal vertex,  the internal edges pointing out of this vertex, except the leftmost and rightmost ones. The forest thus obtained is a system of non-crossing binary trees. Here is our example, with the leaves labelled from left to right.
$$
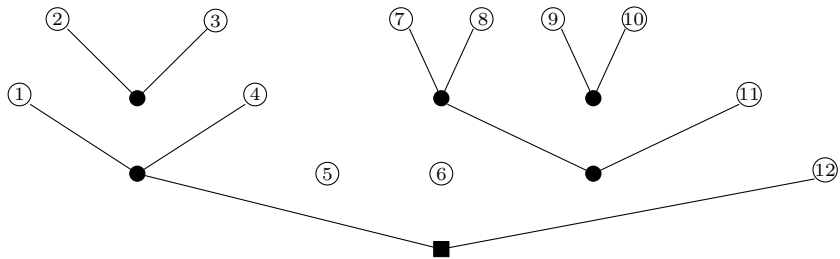
\begin{tikzpicture}[scale=1]
\draw[fill] (5.9,-.1)--(6.1,-.1)--(6.1,.1)--(5.9,.1)-- (5.9,-.1);\draw[fill] (2,1) circle (.1cm);
\draw[fill] (2,2) circle (.1cm);\draw[fill] (8,1) circle (.1cm);\draw[fill] (6,2) circle (.1cm);\draw[fill] (8,2) circle (.1cm);

\draw (4.5,1) circle (.15cm);\node at (4.5,1) {$\scriptstyle 5$};
\draw (6,1) circle (.15cm);\node at (6,1) {$\scriptstyle 6$};
\draw (0.45,2.05) circle (.15cm);\node at (0.45,2.05) {$\scriptstyle 1$};
\draw  (.95,3.05) circle (.15cm);\node at (0.95,3.05) {$\scriptstyle 2$};
\draw (3.03,3.03) circle (.15cm);\node at (3.03,3.03) {$\scriptstyle 3$};
\draw (3.55,2.05) circle (.15cm);\node at (3.55,2.05) {$\scriptstyle 4$};

\draw (5.47,3.05) circle (.15cm);\node at (5.45,3.05) {$\scriptstyle 7$};
\draw (6.53,3.05) circle (.15cm);\node at (6.55,3.05) {$\scriptstyle 8$};
\draw (7.47,3.05) circle (.15cm);\node at (7.47,3.05) {$\scriptstyle 9$};
\draw (8.54,3.05) circle (.16cm);\node at (8.53,3.05) {$\scriptstyle 10$};
\draw (11.05,1.05) circle (.16cm);\node at (11.05,1.05) {$\scriptstyle 12$};
\draw(10.05,2.05) circle (.16cm);\node at (10.05,2.05) {$\scriptstyle 11$};

\draw (2,2)--(1.08,2.92);\draw (2,2)--(2.92,2.92);
\draw (.59,1.92)--(2,1)--(6,0)--(10.91,0.93);
\draw (2,1)--(3.42,1.92);
\draw (8,1)--(6.08,1.92);
\draw (8,1)--(9.92,1.92);
\draw (6,2)--(5.58,2.92);\draw (6,2)--(6.42,2.92);\draw (8,2)--(8.42,2.92);\draw (8,2)--(7.58,2.92);

\end{tikzpicture}
$$

This forest   defines a non-crossing partition of the leaves of the tree such that the first and last leaves are in the same component. 
If we remove this last leaf then this non-crossing partition is  $K(\pi(t))$, the Kreweras complement of $\pi(t)$  (here we use the fact that the Schr\"oder tree is prime). In our example the non-crossing partition of the leaves is 
$$\{1,4,12\}\cup\{2,3\}\cup\{5\}\cup\{6\}\cup\{7,8,11\}\cup\{9,10\}$$ and $K(\pi(t))$ is obtained by removing $12$ from the first part of this partition.

Conversely, given a non-crossing partition $\pi$ of $1,\ldots, n-1$, all prime Schr\"oder trees $t$ such that $\pi(t)=\pi$ are obtained by chosing, for each part of $K(\pi)$, a binary tree whose leaves are numbered by the elements of the part, adding a leaf to the right of the tree containing the leftmost leaf, then joining the binary trees together to form the prime Schr\"oder tree. Since binary trees are counted by Catalan numbers, it follows that the number of prime Schr\"oder trees satifying  $\pi(t)=\pi$ is equal to 
\begin{equation}\label{numbtr}
\prod_{\text{$p$ part of $K(\pi)$}}\text{Cat}_{|p|-1}
\end{equation}where $|p|$ denotes the number of elements in the part $p$. Comparing to (\ref{Mob}) we see   that this number is 
$|\mu(\pi)|$.
See \cite{JMNT} for details about this construction.
\subsubsection{ Cluster complex and associahedra }
Consider a regular polygon drawn in the plane with $n+1$ vertices, enumerated in clockwise order. A pair of vertices which are not adjacent is called a {\sl diagonal}. One can represent the diagonal by a  segment joining the two vertices  inside the polygon. Two diagonals are {\sl compatible} if the segments they determine do not cross inside the polygon (they may cross however at the vertices). The compatibility relation determines a flag simplicial complex called the {\sl cluster complex}. The maximal subsets of pairwise compatible diagonals  determine  triangulations of the polygon with vertices on the boundary while, more generally, compatible sets determine {\sl dissections} of the polygon into smaller polygons. Here is a picture with $n=7$ where the compatible diagonals are $(1,6),(3,6)$ and $(6,8)$. The resulting dissection has two triangles with vertices $\{1,6,8\}$ and $\{6,7,8\}$ and two quadrangles with vertices $\{1,2,3,6\}$ and $\{3,4,5,6\}$.
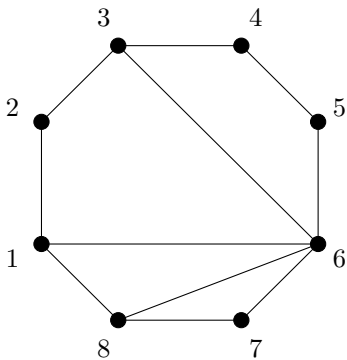
$$
\begin{tikzpicture}[scale=1]
\draw[fill] (0.81,1.82) circle (.1cm);\draw[fill] (-0.81,1.82) circle (.1cm);\draw[fill] (-0.81,-1.82) circle (.1cm);\draw[fill] (0.81,-1.82) circle (.1cm);
\draw[fill] (1.82,0.81) circle (.1cm);\draw[fill] (1.82,-0.81) circle (.1cm);\draw[fill] (-1.82,-0.81) circle (.1cm);\draw[fill] (-1.82,0.81) circle (.1cm);
\node at (-1,2.2){$3$};\node at (1,2.2){$4$};
\node at (2.1,1){$5$};
\node at (-1,-2.2){$8$};\node at (1,-2.2){$7$};\node at (2.1,-1){$6$};
\node at (-2.2,-1){$1$};\node at (-2.2,1){$2$};
\draw (-0.81,1.82)--(0.81,1.82)--(1.82,0.81)--(1.82,-0.81)--(0.81,-1.82)--(-0.81,-1.82)--(-1.82,-0.81)--(-1.82,0.81)--(-0.81,1.82);
\draw (-0.81,1.82)--(1.82,-0.81)--(-1.82,-0.81);\draw (1.82,-0.81)--(-0.81,-1.82);
\end{tikzpicture}
$$
The cluster complex and its dual simplicial complex, the {\sl associahedron}, can be both realized as the boundary of  convex polytopes, see e.g. \cite{FR} for more information and relations with Coxeter combinatorics.

There is a simple bijection between the faces of the cluster complex of a polygon with $n+1$ vertices and  Schr\"oder trees. Take a dissection of a polygon with $n+1$ edges and draw the polygon with the edge $[n,n+1]$ at the base. Put a vertex inside each polygonal face of the dissection, these vertices correspond to internal vertices of a Schr\"oder tree. The vertex inside the face having $[n,n+1]$ in its boundary will be  the root of the tree. Draw internal edges between the vertices across the diagonals of the dissection and external edges leading to leaves across the segments $ [n+1,1]$ and $[i,i+1]$ for $i=1,2,\ldots,n-1$. The resulting tree is a Schr\"oder tree  with $n$ leaves and it is easy to see that this is a bijection between dissections and Schr\"oder trees. Here is the Schr\"oder tree corresponding to the above dissection.
$$
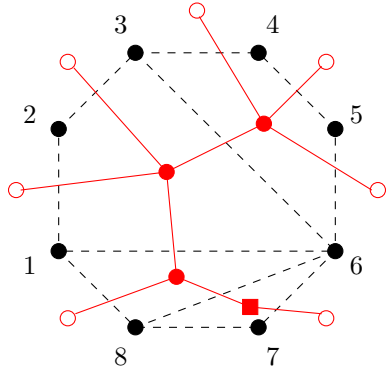
\begin{tikzpicture}[scale=1]
\draw[fill] (0.81,1.82) circle (.1cm);\draw[fill] (-0.81,1.82) circle (.1cm);\draw[fill] (-0.81,-1.82) circle (.1cm);\draw[fill] (0.81,-1.82) circle (.1cm);
\draw[fill] (1.82,0.81) circle (.1cm);\draw[fill] (1.82,-0.81) circle (.1cm);\draw[fill] (-1.82,-0.81) circle (.1cm);\draw[fill] (-1.82,0.81) circle (.1cm);
\node at (-1,2.2){$3$};\node at (1,2.2){$4$};
\node at (2.1,1){$5$};
\node at (-1,-2.2){$8$};\node at (1,-2.2){$7$};\node at (2.1,-1){$6$};
\node at (-2.2,-1){$1$};\node at (-2.2,1){$2$};
\draw[dashed] (-0.81,1.82)--(0.81,1.82)--(1.82,0.81)--(1.82,-0.81)--(0.81,-1.82)--(-0.81,-1.82)--(-1.82,-0.81)--(-1.82,0.81)--(-0.81,1.82);
\draw[dashed] (-0.81,1.82)--(1.82,-0.81)--(-1.82,-0.81);\draw[dashed] (1.82,-0.81)--(-0.81,-1.82);

\draw[fill,color=red] (-0.27,-1.15) circle (.1cm);
\draw[fill,color=red] (-0.4,0.24) circle (.1cm);
\draw[fill,color=red] (0.88,0.88) circle (.1cm);
\draw[color=red] (0.7,-1.55)--(-0.27,-1.15)--(-0.4,0.24)--(0.88,0.88);
\draw[color=red] (-1.7,-1.7) circle (.1cm);\draw[color=red] (-2.38,0) circle (.1cm);\draw[color=red] (-1.7,1.7) circle (.1cm);\draw[color=red] (1.7,-1.7) circle (.1cm);
\draw[color=red] (2.38,0) circle (.1cm);\draw[color=red] (0,2.38) circle (.1cm);\draw[color=red] (1.7,1.7) circle (.1cm);
\draw[color=red] (-1.62,1.62)--(-0.4,0.24);\draw[color=red] (-2.32,0)--(-0.4,0.24);\draw[color=red] (-1.60,-1.64)--(-0.27,-1.15);
\draw[color=red] (.03,2.30)--(0.88,0.88);\draw[color=red] (2.3,0)--(0.88,0.88);\draw[color=red] (1.64,1.64)--(0.88,0.88);\draw[color=red] (0.7,-1.55)--(1.64,-1.64);
\draw[fill,color=red] (0.6,-1.65)--(0.8,-1.65)--(0.8,-1.45)--(0.6,-1.45)--(0.6,-1.65);
\end{tikzpicture}
$$

 Observe that the corners of the tree are in bijection with the vertices $1,2,\ldots,n-1$ of the polygon. The Schr\"oder tree is prime if and only if the vertex $n$ does not belong to one of the diagonals of the dissection.
Using this bijection we could rephrase all the constructions of the next section in terms of  the cluster complex or the associahedra. However we think that using Schr\"oder trees makes our constructions easier to understand.
\section{The formula}

In this section I give an explicit formula for the polynomials $Q_\sigma$.
For a prime Schr\"oder tree $t$, with $n+1$ leaves, an integer $k\in[1,n]$ and a circular permutation $\sigma$, label the corners of $t$, from left to right, by the numbers $\sigma(k),\sigma^2(k),\ldots,\sigma^{n-1}(k),k$.  For each internal vertex $v$ of $t$ let $i(v)$ be the smallest label of all corners belonging to $v$. Define $x^{t,k,\sigma }$ as the product of $-x_{i(v)}$ over all internal vertices of $t$. Let us consider again the Schr\"oder tree of section 4, and the cycle $\sigma= (2,4,8,5,9,1,6,11,10,3,7)$, with $k=7$:

$$
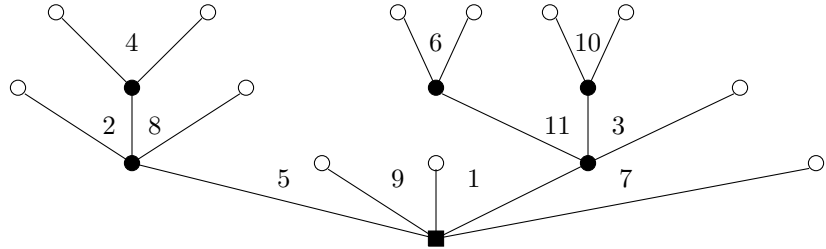
\begin{tikzpicture}[scale=1]
\draw[fill] (5.9,-.1)--(6.1,-.1)--(6.1,.1)--(5.9,.1)-- (5.9,-.1);\draw[fill] (2,1) circle (.1cm);
\draw[fill] (2,2) circle (.1cm);\draw[fill] (8,1) circle (.1cm);\draw[fill] (6,2) circle (.1cm);\draw[fill] (8,2) circle (.1cm);

\draw (4.5,1) circle (.1cm);\draw (6,1) circle (.1cm);\draw (11,1) circle (.1cm);
\draw (0.5,2) circle (.1cm);\draw (3.5,2) circle (.1cm);\draw(10,2) circle (.1cm);
\draw (1,3) circle (.1cm);\draw (3,3) circle (.1cm);\draw (5.5,3) circle (.1cm);\draw (6.5,3) circle (.1cm);\draw (7.5,3) circle (.1cm);\draw (8.5,3) circle (.1cm);

\draw (2,2)--(1.08,2.92);\draw (2,2)--(2.92,2.92);
\draw (.59,1.92)--(2,1)--(6,0)--(10.91,0.93);\draw (2,1)--(2,2);\draw (6,0)--(8,1);
\draw (2,1)--(3.42,1.92);\draw (6,0)--(4.58,.92);\draw (6,0)--(6,.92);
\draw (8,1)--(6.08,1.92);\draw (8,1)--(8,1.92);\draw (8,1)--(9.92,1.92);
\draw (6,2)--(5.58,2.92);\draw (6,2)--(6.42,2.92);\draw (8,2)--(8.42,2.92);\draw (8,2)--(7.58,2.92);

\node at (1.7,1.5) {$2$};\node at (2,2.6) {$4$};\node at (2.3,1.5) {$8$};\node at (4,.8) {$5$};\node at (5.5,.8) {$9$};\node at (6.5,.8) {$1$};\node at (6,2.6) {$6$};
\node at (7.6,1.5) {$11$};\node at (8,2.6) {$10$};\node at (8.4,1.5) {$3$};\node at (8.5,0.8) {$7$};
\end{tikzpicture}
$$
The monomial associated to this tree is then 
$$ x^{t,7,\sigma}=(-x_2)(-x_4)(-x_1)(-x_6)(-x_3)(-x_{10})$$
\begin{theorem}\label{Th}
For each $k\in[1,n]$ one has  
\begin{equation}
\label{formuleSchroder}Q_\sigma(x_1,\ldots,x_n)=-\sum_{t\in p\cS_{n+1}}x^{t,k,\sigma }
\end{equation}
\end{theorem}
\begin{proof}Denote by $R^k_\sigma(x_1,\ldots,x_n)$ the polynomial on the right hand side of (\ref{formuleSchroder}).  First note that, by construction, it is of degree one in each variable. 
Also, verifiying the formula for $n=1,2$  is a trivial exercise.

 We now check that $R^k_\sigma$  depends only on the cyclic permutation $\sigma$ and not on $k$. For this observe that the term $x^{t,k,\sigma}$ can be expressed in terms of the partition $\pi(t)$. Using (\ref{numbtr}) and  the remark following this equation, we can rewrite the definition of $R^k_\sigma$ as 

\begin{equation}\label{R_cum}
R^k_\sigma(x_1,\ldots,x_n)=-\sum_{\pi\in NC(n)}\prod_{p\in \pi}(-x_{i(p)})|\mu(K(\pi))|
\end{equation}
Here, for each part $p$ of $\pi$, we let $i(p)$ denote the minimum of the numbers $\sigma^{i_1}(k),\sigma^{i_2}(k),\ldots,\sigma^{i_{|p|}}(k)$ where $p=\{i_1,\ldots, i_{|p|}\}$. 
We will interpret this formula as a free cumulant in the next section but for now 
 note that it implies that the polynomial on the RHS  depends only on the circular permutation $\sigma$ and not on the chosen $k$, since the factor $|\mu(K(\pi))|$ is invariant by a circular permutation of $\pi$. I will therefore denote this polynomial by $R_\sigma$.
 
Let us now check  that these polynomials satisfy the boundary conditions. Since 1 is always the smallest corner belonging to its vertex, the variable $x_1$ is a factor of every term in the sum defining $R_\sigma$, therefore $x_1$ divides $R_\sigma$.
Since the polynomial $R_\sigma$ does not depend on $k$, we can take $k=n$ in the formula.  Each prime Schr\"oder tree is either of the form $t_1$ or $t_2$ as in the construction of section 4.2.2 for some Schr\"oder tree $t$. The corresponding term in the first case is a monomial in $x_1,\ldots,x_{n-1}$ (since $n$ is certainly larger than all the labels of all corners belonging to the root) and in the second case it is the same monomial multiplied by $-x_n$ therefore the polynomial $R_\sigma$ has $1-x_n$ as a factor.

It remains to check the continuity and the exchange conditions (4) and (5). For this we will assume that $k=i+1$ when computing $R_\sigma$ and that $k=i$ when computing $R_{s_i\sigma s_i}$. Condition (4) is immediate to check from the definition so we just need to check the exchange condition. For this let $i,i+1$ be two indices and expand
$$R_\sigma=A+x_iB+x_{i+1}C+x_{i}x_{i+1}D,\quad R_{s_i\sigma s_i}=A'+x_iB'+x_{i+1}C'+x_{i}x_{i+1}D'$$ where $A,A',B,B',C,C',D,D'$ are polynomials in the other variables $x_1,x_2,\ldots,x_{i-1},x_{i+2},\ldots, x_n$.
By the continuity condition one has $A'=A,B'+C'=B+C$ and $D'=D$. Consider now the set of all prime Schr\"oder trees $t$ such that the labels $i$ and $i+1$ (with $t$ labelled by $\sigma$) belong to different vertices while $i$ is the smallest label in its vertex and $i+1$ is not. In the sum 
(\ref{formuleSchroder}) these trees contribute to $B$ for $R_\sigma$ and to $C'$ for $R_{s_i\sigma s_i}$, when the labelling is done according to $s_i\sigma s_i$. Similarly if we exchange the roles of $i$ and $i+1$, they contribute to respectively to $C$ and to $B'$. The value of $B-C'=B'-C$ is obtained by taking the sum  over the remaining trees $t$, for which $i$ and $i+1$ label corners which belong to the same vertex and $i$ is the smallest label at this vertex. Such trees contribute both to $B$ and to $B'$. They have the form

$$
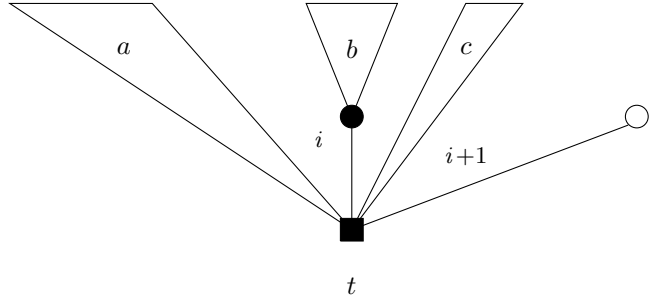
\begin{tikzpicture}[scale=.75]
\draw[fill] (2.8,-.2)--(2.8,.2)--(3.2,.2)--(3.2,-.2)--(2.8,-.2);
\draw (3,0)--(-3,4)--(-.5,4)--(3,0);
\node at (3,3.2) {$b$};
\draw[fill] (3,2) circle (.2cm);
\draw (3,0)--(3,2);
\draw (3,2)--(2.2,4)--(3.8,4)--(3,2);
\draw (3,0)--(5,4)--(6,4)--(3,0);
\draw (8,2) circle (.2cm);
\draw (3,0)--(7.85,1.85);
\node at (-1,3.2) {$a$};
\node at (5,3.2) {$c$};
\node at (2.4,1.6) {$\text {\small \it i}$};
\node at (5,1.3) {$ \text{\small{\it i}\,+1}$};
\node at (3,-1) {$t$};
\end{tikzpicture}
$$
for some trees $a,b,c$. I have shown the corners labelled $i$ and $i+1$, which belong to the root.
We can cut such a tree into two trees by cutting through the edge at the right of the corner labelled $i$. Then we exchange the labels $i+1$ and $i$ to get
two trees $t^+$ and $t^-$  as below
$$
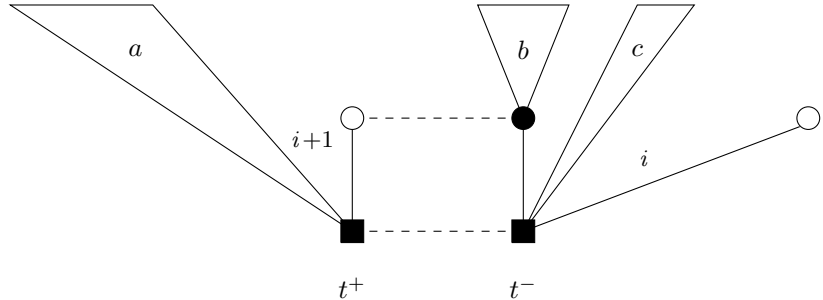
\begin{tikzpicture}[scale=.75]
\draw[fill] (1.8,-.2)--(1.8,.2)--(2.2,.2)--(2.2,-.2)--(1.8,-.2);
\draw[fill] (4.8,-.2)--(4.8,.2)--(5.2,.2)--(5.2,-.2)--(4.8,-.2);
\draw (2,0)--(-4,4)--(-1.5,4)--(2,0);
\node at (5,3.2) {$b$};
\draw (2,2) circle (.2cm);
\draw[fill] (5,2) circle (.2cm);
\draw (2,0)--(2,1.8);
\draw (5,0)--(5,2);
\draw (5,2)--(4.2,4)--(5.8,4)--(5,2);
\draw (5,0)--(7,4)--(8,4)--(5,0);
\draw (10,2) circle (.2cm);
\draw (5,0)--(9.85,1.85);
\node at (-1.8,3.2) {$a$};
\node at (7,3.2) {$c$};
\node at (1.3,1.6) {$ \text{\small{\it i}\,+1}$};
\node at (7.1,1.3) {$ \text{\small\it i}$};
\draw[dashed] (2.3,0)--(4.7,0);
\draw[dashed] (2.3,2)--(4.7,2);
\node at (2,-1) {$t^+$};
\node at (5,-1) {$t^-$};
\end{tikzpicture}
$$
The tree $t^+$  gives a contribution to the formula (\ref{formuleSchroder}) defining $[x_{i+1}]R_{\sigma^+}(x^+)$ while $t^-$ gives a contribution to $[x_i]R_{\sigma^-}(x^-)$, if we compute these polynomials by using the ordering of the cycles $\sigma^-$ and $\sigma^+$ which put $i$ and $i+1$ at the end. Conversely any pair of such trees can be combined to form a tree $t$ as above. It follows that $B'-C=B-C'=\Delta=([x_{i+1}]R_{\sigma^+}(x^+))([x_i]R_{\sigma^-}(x^-))$..
This proves  that the polynomials $R_\sigma$ satisfy the exchange relations therefore, by uniqueness, one has $Q_\sigma=R_\sigma$, as claimed.
\end{proof}

\section{The polynomials $Q_\sigma$ as free cumulants}
Let us  revert to the original interpretation of the polynomials $Q_\sigma$, where the $x_i$ form an increasing subset of $[0,1]$ (see section \ref{model}).
Consider the interval $[0,1]$ equipped with Lebesgue measure and introduce the indicator functions $\Pi_x:=1_{[0,x]}$. We consider them as random variables on the probability space $[0,1]$. Notice that they satisfy the relations:
$$\Pi_x\Pi_y=\Pi_{x\wedge y}.$$
In particular, denoting $\varphi$ the integration with respect to Lebesgue measure we have
\begin{equation}\label{min}
\varphi(\Pi_{u_1}\Pi_{u_2}\ldots\Pi_{u_p})=\min(u_1,u_2,\ldots,u_p).
\end{equation}
\begin{theorem}\label{Th2}
For each $k\in[1,n]$ and $0\leq x_1\leq x_2\ldots\leq x_n\leq 1$ one has
\begin{equation}
\label{formulecumulant}Q_\sigma(x_1,\ldots,x_n)=\kappa_n(\Pi_{x_{\sigma(k)}},\ldots,\Pi_{x_{\sigma^{n-1}(k)}},\Pi_{x_{k}})
\end{equation}
\end{theorem}
\begin{proof}
Using formulas (\ref{cumulant-moment}), (\ref{R_cum}) and (\ref{min}) it is enough to check the signs, namely that 
$$\prod_{\text{$p$ part of $K(\pi)$}}(-1)^{|p|-1}=-\prod_{\text{$p$ part of $\pi$}}(-1)$$ This follows from $\sum_{\text{$p$ part of $K(\pi)$}}|p|=n$ and the well known fact that
$|K(\pi)|+|\pi|=n+1$ ($|\pi|$ denotes the number of parts of $\pi$), see \cite{NS}.

Alternatively we can directly prove that the right hand side of (\ref{formulecumulant}) satisfies all the properties of Definition \ref{loop}.
Since it is very similar to the proof of Theorem \ref{Th} I only sketch the argument.
Let us call $S_\sigma$ this right hand side (it does not depend on $k$ by the cyclic invariance  of free cumulants). Property  $(1)$ is obvious from the moment-cumulant formula (\ref{cumulant-moment}) and  (\ref{min}) while $(2)$ follows from a simple computation. Relations $(3)$ follow from the fact that a free cumulant vanishes if one of its entries is a constant while $(4)$ just states that $\kappa_p(\Pi_{u_1},\ldots\Pi_{u_p})$ is a continuous function of the $u_j$ on the cube $[0,1]^p$. Finally it remains to check the exchange condition $(5)$.
For this let $\sigma$ be a cyclic permutation and write $S_\sigma $ and   $S_{s_i\sigma s_i}$ as sums over non-crossing partitions, using (\ref{cumulant-moment}) and (\ref{min}). Let  $\pi$ be a non-crossing partition of $[1,n]$ and assume that in the corresponding term of
(\ref{cumulant-moment}) the  $x_i$ and $x_{i+1}$ are in  different parts of $\pi$, then it is easy to see that the corresponding terms have the same value in 
$S_\sigma$ and $S_{s_i\sigma s_i}$. Suppose  now that $x_i$ and $x_{i+1}$ are in the same part $p$ of $\pi$, corresponding to the following picture where we show only the part $p$.

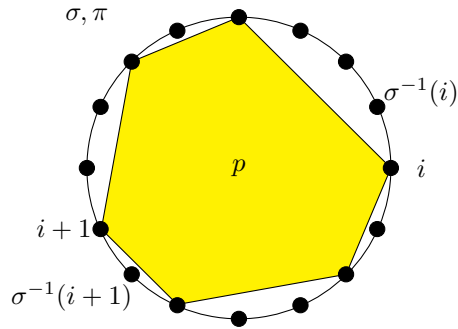
$$
\begin{tikzpicture}[scale=1]
\draw[fill=yellow] (0,2)--(2,0)--(1.41,-1.41)--(-.81,-1.81)--(-1.81,-.81)--(-1.41,1.41)--(0,2);
\draw (0,0) circle (2cm);
\draw[fill] (0,2) circle (.1cm);\draw[fill] (0,-2) circle (.1cm);\draw[fill] (2,0) circle (.1cm);\draw[fill] (-2,0) circle (.1cm);\draw[fill] (1.41,1.41) circle (.1cm);
\draw[fill] (1.41,-1.41) circle (.1cm);\draw[fill] (-1.41,1.41) circle (.1cm);\draw[fill] (-1.41,-1.41) circle (.1cm);

\node at (0,0){$p$};
\node at (2.4,0){$i$};
\node at (-2.3,-.81){$i+1$};
\node at (-2.2,-1.7){$\sigma^{-1}(i+1)$};
\node at (2.4,1){$\sigma^{-1}(i)$};
\node at (-2,2){$\sigma,\pi$};
\draw[fill] (0.81,1.82) circle (.1cm);
\draw[fill] (-0.81,1.82) circle (.1cm);
\draw[fill] (-0.81,-1.82) circle (.1cm);\draw[fill] (0.81,-1.82) circle (.1cm);
\draw[fill] (1.82,0.81) circle (.1cm);\draw[fill] (1.82,-0.81) circle (.1cm);
\draw[fill] (-1.82,-0.81) circle (.1cm);
\draw[fill] (-1.82,0.81) circle (.1cm);

\end{tikzpicture}
$$
When we multiply $\sigma$ on the left by $s_i$ we obtain two cycles  $\sigma ^+$ and  $\sigma^-$ and we can accordingly cut the part $p$ into two parts $p^+$ and $p^-$. Each other part of $\pi$ is either included in $\sigma ^+$ or in  $\sigma^-$ so that  we   get two noncrossing partitions, $\pi^+$ and $\pi^-$, of $\sigma ^+$ and  $\sigma^-$.
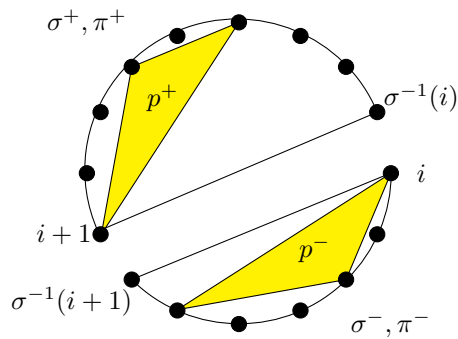
$$
\begin{tikzpicture}[scale=1]
\draw[fill=yellow] (0,2)--(-1.41,1.41)--(-1.81,-.81)--(0,2);
\draw[fill=yellow] (1.41,-1.41)--(-.81,-1.81)--(2,0)--(1.41,-1.41);

\draw  (1.82,0.81) arc[start angle=22.5, end angle=204,radius=2cm]--(1.82,0.81);
\draw  (-1.41,-1.41) arc[start angle=225, end angle=360,radius=2cm]--(-1.41,-1.41);

\draw[fill] (0,2) circle (.1cm);\draw[fill] (0,-2) circle (.1cm);\draw[fill] (2,0) circle (.1cm);\draw[fill] (-2,0) circle (.1cm);\draw[fill] (1.41,1.41) circle (.1cm);
\draw[fill] (1.41,-1.41) circle (.1cm);\draw[fill] (-1.41,1.41) circle (.1cm);\draw[fill] (-1.41,-1.41) circle (.1cm);

\node at (2.4,0){$i$};
\node at (-2.3,-.81){$i+1$};
\node at (-2.2,-1.7){$\sigma^{-1}(i+1)$};
\node at (2.4,1){$\sigma^{-1}(i)$};
\node at (-2,2){$\sigma^{+},\pi^+$};
\node at (2,-2){$\sigma^{-},\pi^-$};
\draw[fill] (0.81,1.82) circle (.1cm);
\draw[fill] (-0.81,1.82) circle (.1cm);
\draw[fill] (-0.81,-1.82) circle (.1cm);\draw[fill] (0.81,-1.82) circle (.1cm);
\draw[fill] (1.82,0.81) circle (.1cm);\draw[fill] (1.82,-0.81) circle (.1cm);
\draw[fill] (-1.82,-0.81) circle (.1cm);
\draw[fill] (-1.82,0.81) circle (.1cm);
\node at (-1,1){$p^+$};
\node at (1,-1){$p^-$};
\end{tikzpicture}
$$

It is then easy to see that, as in the proof of Theorem \ref{Th}, the term in factor of $x_i$ corresponding to the partition $\pi$ is the product of the terms in factor of $x_i$ and $x_{i+1}$ in $S_\sigma$ and $S_{s_i\sigma s_i}$, respectively, corresponding to the partitions $\pi^-$ and $\pi^+$. 
\end{proof}

\end{document}